\documentclass[12pt, oneside, a4paper, reqno]{amsart}
\usepackage[foot]{amsaddr}
\usepackage{amsfonts, amssymb, dsfont, nicefrac,bm,upgreek,mathtools,verbatim, color, amsopn,esint}
\usepackage[final]{hyperref}
\usepackage[mathscr]{eucal}
\usepackage[normalem]{ulem}
\usepackage[T1]{fontenc}
\hypersetup{linkcolor=blue, colorlinks=true
,citecolor = red}
\allowdisplaybreaks
\usepackage{setspace}
\usepackage[normalem]{ulem}
\usepackage{cancel}
\usepackage{soul}
\usepackage[left=2.75cm,right=2.75cm,top=2.6cm,bottom=2.6cm]{geometry}
\usepackage[hyperpageref]{backref}
\newtheorem{theorem}{Theorem}[section]
\newtheorem{lemma}[theorem]{Lemma}

\newtheorem{proposition}[theorem]{Proposition}
\theoremstyle{definition}
\newtheorem{definition}[theorem]{Definition}
\newtheorem{remark}[theorem]{Remark}

\numberwithin{theorem}{section}
\numberwithin{equation}{section}

\newcommand{\D}{\mathcal{D}}
\newcommand{\R}{\mathbb{R}}
\newcommand{\w} {\rightharpoonup}
\newcommand{\overbar}[1]{\mkern 1.5mu\overline{\mkern-1.5mu#1\mkern-1.5mu}\mkern 1.5mu}
\newcommand\numberthis{\addtocounter{equation}{1}\tag{\theequation}}

\usepackage[capitalize,nameinlink]{cleveref}

\newcommand{\ttl}{\Large Principal eigenvalues for the weighted p-Laplacian and antimaximum principle in $\R^N$}

\begin{document}
\title[Weighted p-Laplacian in $\R^N$]
{\ttl}

\author[as]{Anumol Joseph$^{1,*}$} 
\address{$*$Corresponding author} 
\address{$^1$Centre for Applicable Mathematics, Tata Institute of Fundamental Research, Post Bag No 6503, Sharadanagar, Bangalore - 560065, India.}
\email{anumol24@tifrbng.res.in}

\author{Abhishek Sarkar$^2$}
\address{$^2$Department of Mathematics, Indian Institute of Technology Jodhpur, NH 65, Jodhpur, Rajasthan - 342030, India.}
\email{abhisheks@iitj.ac.in}

\begin{abstract} 
We study the existence of principal eigenvalues and principal eigenfunctions for weighted eigenvalue problems of the form:
\begin{equation*}
- \mbox{div} ( L (x) |\nabla u|^{p-2} \nabla u ) = \lambda K(x) |u|^{p-2} u \hspace{.1cm} \mbox { in } \hspace{.1cm} \mathbb{R}^N ,
\end{equation*} 
where $\lambda \in \mathbb{R}$, $p>1$, $K : \mathbb{R}^N \rightarrow \mathbb{R}$, $L : \mathbb{R}^N \rightarrow \mathbb{R}^+$ are locally integrable functions. The weight function $K$ is allowed to change sign, provided it remains positive on a set of nonzero measure. We establish the existence, regularity, and asymptotic behavior of the principal eigenfunctions. We also prove local and global antimaximum principles for a perturbed version of the problem.
\end{abstract}


\maketitle

\section{Introduction}
In this article, we study the existence of principal eigenvalues and principal eigenfunctions for weighted eigenvalue problems of the form:
\begin{equation} \label{Pblm 1}
- \mbox{div} ( L (x) |\nabla u|^{p-2} \nabla u ) = \lambda K(x) |u|^{p-2} u \hspace{.1cm} \mbox { in } \hspace{.1cm} \mathbb{R}^N ,
\end{equation} 
where $\lambda \in \mathbb{R}$, $p>1$, $K : \mathbb{R}^N \rightarrow \mathbb{R}$, $L : \mathbb{R}^N \rightarrow \mathbb{R}^+$ are locally integrable functions. We allow the weight function $K$ to change sign but assume that $K$ is positive on a set of nonzero measures (more assumptions will be followed).

Weighted eigenvalue problems of this form have been extensively studied for the $p$-Laplace and Laplace operators on various bounded and unbounded domains. Several results in the literature provide sufficient conditions on the weight function $K$ for the existence or nonexistence of principal eigenvalues. Both cases $1<p<N$ and $p \geq N$ have been well explored in the case of bounded domains and exterior domains in $\mathbb{R}^N$ (see \cite{allegretto1992principal}, \cite{allegretto1995eigenvalues}, \cite{Anoop}, \cite{Anoop-Drabek-Sarath}, \cite{brown},  \cite{Drabek-95}, \cite{Drabek-Abhishek}, \cite{Huang}, \cite{Manes-Micheletti}, \cite{Rozenblum-Solomyak}, \cite{Stavrakakis}, \cite{Szulkin}, \cite{LaoSenYu} and the references therein). But when it comes to general unbounded domains in $\mathbb{R}^N$, the case $p \geq N$ becomes significantly challenging. There are two main difficulties, one
arising from the choice of solution space and the other from the choice of weight function. The natural solution space for such problems (for the case $L\equiv 1$) is $\mathcal{D}_0^{1,p}(\mathbb{R}^N)$, which is the completion of $C_c^{\infty}(\mathbb{R}^N)$ with respect to $\| \nabla u \|_p = (\int_{\mathbb{R}^N} |\nabla u(x)|^p dx)^{1/p}$,  cannot be identified as a function space since it contains objects that do not lie even in $L_{loc} ^1 (\mathbb{R}^N)$ (for example, see \cite[Remark 2.2]{tintarev_fieseler_2007}). In \cite{brown}, Brown et. al. proved that, for $p = N=2$ and $L(x) \equiv 1$, if $\int_{\mathbb{R}^2} K(x) dx >0$, then the problem (\ref{Pblm 1}) has no solution in $W^{1,2} (\mathbb{R}^2)$. However, it remains unclear whether the condition $\int_{\mathbb{R}^2} K(x) dx\leq 0$ is sufficient for the existence of principal eigenvalues. There has been some progress in understanding this, particularly in establishing the existence when the weight function $K$ has a dominant negative part. For an overview, we refer the reader to \cite{allegretto1992principal}, \cite{allegretto1995eigenvalues}, \cite{brown}, \cite{Huang}, \cite{Rozenblum-Solomyak}. Another significant contribution to a related problem was made by Anoop et al. in \cite{Anoop-Drabek-Sarath}. They studied the problem (\ref{Pblm 1}) with $L(x) \equiv 1$ on the exterior of a ball in $\R^N$ and established the existence of positive eigenvalues for weights $K$ dominated by radial functions in certain weighted Lebesgue spaces,  for both $1<p<N$ and $p \geq N$. In contrast to the non-existence results in $\R^N$ for $p \geq N$, they proved existence results in the exterior domains even for some positive weights.

In this article, we study the problem (\ref{Pblm 1}) in $\R^N$ and establish principal eigenfunctions' existence, regularity, and asymptotic properties. Note that we have no restriction on the dimension $N$ in terms of $p$. We begin by stating the precise assumptions on the weight functions  $K$  and  $L$. Specifically, we assume the existence of functions $v$ and $w$ such that $v,w \in L_{loc}^1(0,\infty)$, $v(|x|) > |x| ^{-p \alpha}$, $0<w(|x|) < |x|^{- p \beta}$ where $-1< \alpha < 0, \beta = \alpha+1$ and 
\begin{enumerate}
\item[(C1)] $L(x) \geq v(|x|)$ and $|K(x)| \leq w(|x|)$. 
\item[(C2)] For any $r \in (0, \infty)$, $G(r):= \left(\int_r^\infty t^{\frac{1-N}{p-1}} v(t)^{\frac{-1}{p-1}} dt\right)^{p/p'}$ and $\int_0^\infty r^{N-1} w(r) G(r) dr \\< \infty$, where $p'=\frac{p}{p-1}$. 
\end{enumerate}
\begin{remark}
An example of weight functions  $v$  and  $w$  satisfying our assumptions is given as follows: 
\begin{itemize}
    \item $v(r) = r^{-p \alpha + \zeta} $ with $\zeta > \max\{1, p(\alpha+1)-N)\}$ and
    \item $w(r) = \begin{cases}
    \frac{1}{r^{p\beta}-\beta} ; 0<r \leq 1 ,\\
    \frac{1}{r^{p\beta+1}} ; r>1.
\end{cases}$
\end{itemize} 
\end{remark}

We now define an appropriate notion of a solution space for our problem. For $u \in C_0^\infty(\R^N)$, we define the norm, 
\begin{equation*}
\| u \|_L = \int_{\R^N} L(x) |\nabla u|^p dx,
\end{equation*}
and introduce the space $\mathcal{D}_L^{1,p}(\mathbb{R}^N)$ as the completion of $C_0^\infty(\mathbb{R}^N)$ with respect to $\| \cdot \|_L$. It follows that $\mathcal{D}_L^{1,p}(\mathbb{R}^N)$ is an inner product space with respect to the inner product 
\begin{equation*}
(u,v) = \int_{\mathbb{R}^N} L(x) |\nabla u|^{p-2} \nabla u  \cdot \nabla v dx , \hspace{0.35 cm} u , v \in \mathcal{D}_L^{1,p}(\mathbb{R}^N).
\end{equation*}

As a consequence of the Caffareli, Kohn and Nirenberg(C-K-N, in short) inequality \cite{MR0768824}, we find that $\mathcal{D}_L^{1,p}(\mathbb{R}^N)$ is continuously embedded in $L_K^p(\mathbb{R}^N) = \big\{u:\mathbb{R}^N \rightarrow \mathbb{R} : \int_{\mathbb{R}^N} |K(x)| |u|^p dx < \infty  \big\}$. In addition, one can prove that this embedding is compact. For completeness, we will provide proof of this result in Section \ref{Prelim}.

We now define the notion of a (weak) solution:
\begin{definition}
We say $u \in \mathcal{D}_L^{1,p}(\mathbb{R}^N)$ is a (weak) solution of the problem (\ref{Pblm 1}) if $u$ satisfies \begin{equation} \label{weak solution}
\int_{\mathbb{R}^N} L(x) |\nabla u|^{p-2} \nabla u \nabla v dx = \lambda \int_{\mathbb{R}^N} K(x) |u|^{p-2} u v dx , \mbox{ for all }  v \in \mathcal{D}_L^{1,p}(\mathbb{R}^N).
\end{equation}
\end{definition}

A real number $\lambda$ is called an eigenvalue of (\ref{Pblm 1}) if there exists a $u \in \mathcal{D}_L^{1,p}(\mathbb{R}^N)$, $u \neq 0$ such that $u$ is a weak solution to the problem (\ref{Pblm 1}). The function $u$ is called an eigenfunction corresponding to $\lambda$. Below we state our results in Theorems \ref{existence}--\ref{Theorem: Global AMP}. 
\begin{theorem}\label{existence}
Let $L: \mathbb{R}^N \rightarrow \mathbb{R}^+$ and $K: \mathbb{R}^N \rightarrow \mathbb{R}$ be locally integrable functions satisfying conditions (C1)-(C2). Then there exists a principal eigenvalue of (\ref{Pblm 1}) with a corresponding principal eigenfunction that is positive a.e. in $\mathbb{R}^N$.
\end{theorem}
We also study the regularity and asymptotic properties of the principal eigenfunction. We have the following result.
\begin{theorem}\label{regularity}
Let $K$ and $L$ be as stated in Theorem \ref{existence}. If $K \in L^{\infty}(\mathbb{R}^N)$ and $L \in C_{loc}^1(\R^N)$, then $u \in C_{loc}^{1, \alpha}(\mathbb{R}^N)$ for some $\alpha \in (0,1)$.
\end{theorem}
Unlike the case for principal eigenfunctions on unbounded domains in $\mathbb{R}^N$, where $1< p < N$, the behavior is different when $p = N =2$. In this case, the principal eigenfunctions do not decay to zero at infinity (see \cite[Proposition 2.6]{MR1849197}). Instead, they are bounded away from zero. We prove that under certain additional assumptions on the weight functions, the eigenfunctions are bounded. 
\begin{theorem}\label{asymptotic}
Let $p=N=2$ and let $K$ and $L$ satisfy the conditions of Theorem \ref{existence}. If $K$ is positive and both $K$ and $L$ are radial, then the principal eigenfunction $u$ of (\ref{Pblm 1}) is radially increasing and does not decay to zero. Furthermore, if $(r L(r))^{-1} \in L^1(0, \infty)$ and $K$ satisfy $\int_0^\infty s [F(s)]^2 K(s) ds < \infty$, where $F(r) = \int_0^r \frac{1}{sL(s)} ds$, then $u$ is bounded. 
\end{theorem}
Next, we study an antimaximum principle for the problem
\begin{equation}\label{Pblm AMP}
- \mbox{div} ( L (x) |\nabla u|^{p-2} \nabla u ) = \lambda K(x)  |u|^{p-2} u \hspace{.1cm} + h(x) \mbox { in } \hspace{.1cm} \R^N.
\end{equation}
Here, the weight function $L$ is nonnegative, $K$ is allowed to change sign and $h \geq 0$, $h \not \equiv 0$ lies in $[\D_L^{1,p}(\R^N)]^*$, the dual space of $\D_L^{1,p}(\R^N)$. We know from the strong maximum principle that, for a bounded domain $\Omega \subset \R^N$ and $h \in L^2(\Omega)$ with $h \geq 0$, $h \not \equiv 0$, any solution to the problem
\begin{equation}\label{Pblm Bounded}
- \Delta u = \lambda  u \hspace{.1cm} + h(x) \mbox { in } \hspace{.1cm} \Omega, \\
u = 0 \hspace{.1cm} \mbox { on } \hspace{.1cm} \partial \Omega,
\end{equation}
is strictly positive when $\lambda < \lambda_1$. Here, $\lambda_1$ is the first eigenvalue of the Laplacian on $\Omega$ with Dirichlet boundary conditions. In \cite{clement}, Clement and Peletier showed that for certain values of $\lambda > \lambda_1$, the behavior is entirely contrary to that given by the maximum principle. Specifically, they proved that, given $ h(x) \geq 0, \not\equiv 0 $, there exists a $\delta > 0 $ such that if $\lambda  \in (\lambda_1 , \lambda_1 + \delta )$ and $u$ is a solution to (\ref{Pblm Bounded}) at $\lambda$, then $ u(x) < 0 $ for all $x \in \Omega $. This phenomenon is known as the antimaximum principle. In this article, we establish a local antimaximum principle for the problem (\ref{Pblm AMP}). Specifically, we prove that for any bounded set $E \subset \R^N$, and $h \geq 0$, $h \not \equiv 0$, there exists a $\delta = \delta(h, E)>0$ such that any solution $u$ of (\ref{Pblm AMP}) is negative in $E$ for every $\lambda \in (\lambda_1, \lambda_1+\delta)$. Furthermore, we show that the global antimaximum principle holds when $h$ has compact support. That is, there exists a $\delta= \delta(h)$ such that $u$ is negative in $\R^N$ for every $\lambda \in (\lambda_1, \lambda_1+\delta)$. 

We now state the main results of this article. Let the weight functions $K,L$ be stated in Theorem \ref{existence} and $h \geq 0$, $h \not \equiv 0$. In addition, let $K,L,h \in L^s(\R^N) \cap L_{loc}^\infty(\R^N)$ for some $s>\frac{N}{p}$. We say that $u \in \D_L^{1,p}(\R^N)$ is a (weak) solution to (\ref{Pblm AMP}) if $u$ satisfies 
\begin{equation*}
\int_{\mathbb{R}^N} L(x) |\nabla u|^{p-2} \nabla u \nabla v dx = \lambda \int_{\mathbb{R}^N} K(x) |u|^{p-2} u v dx + \int_{\R^N} h(x) v dx , \ \forall v \in \mathcal{D}_L^{1,p}(\mathbb{R}^N).
\end{equation*}
Using the assumptions on $K,L$ and $h$ along with the standard regularity theory, we see that a weak solution, if it exists, is in $C_{loc}^{1,\alpha}(\R^N)$, for some $\alpha \in (0,1)$. The following results establish the local and global antimaximum principles.
\begin{theorem} \label{Theorem: local AMP}
Let $K,L$ be as stated in Theorem \ref{existence} and $h$ satisfy $h \geq 0$, $h \not \equiv 0$. Assume in addition that $K,L,h \in L^s(\R^N) \cap L_{loc}^\infty(\R^N)$. Then for any bounded set $E \subset \R^N$, there exists $\delta = \delta(h, E)>0$ such that if $\lambda \in (\lambda_1, \lambda_1+\delta)$, any solution $u$ of (\ref{Pblm AMP}) is negative in $E$.
\end{theorem}

\begin{theorem}\label{Theorem: Global AMP}
Let $K,L$ and $h$ be as in Theorem \ref{Theorem: local AMP}. If $h$ has compact support, then the global antimaximum principle holds for (\ref{Pblm AMP}).
\end{theorem}

The outline of this article is as follows. Section \ref{Prelim} introduces and recalls the necessary tools and terminologies and more on the solution space and its embeddings. The existence and properties of the principal eigenfunctions are established in Section \ref{Sec3}. Finally,  in Section \ref{Sec4}, we study the antimaximum principle for the perturbed problem.

\section{Preliminaries}\label{Prelim}
Throughout the article, we denote generic constants by $C$. Note that $C$ may represent different constants within the same proof. Let us first recall the celebrated C-K-N inequality \cite{MR0768824}.  
\begin{lemma} \label{C_K_N 1}
For all $u \in C_0^\infty(\mathbb{R}^N)$ and some constant $C>0$,
\begin{equation}
\int_{\mathbb{R}^N} |x|^{- p \beta} |u|^p dx \leq C \int_{\mathbb{R}^N} |x|^{-p \alpha} |\nabla u|^p dx,
\end{equation}
where $-1< \alpha < 0, \beta = \alpha+1$.
\end{lemma}	
As mentioned in the previous section, an immediate consequence of the C-K-N inequality is that $\mathcal{D}_L^{1,p}(\mathbb{R}^N)$ is continuously embedded in $L_K^p(\mathbb{R}^N)$. We now prove that this embedding is compact.
\begin{proposition} \label{Compactness}
Assume that the functions $L$, $K$ satisfy assumptions (C1) and (C2). Then for any $u \in \mathcal{D}_L^{1,p}(\mathbb{R}^N)$, 
\begin{equation} \label{Equation 5} 
\| u \| _{L_K^p(\mathbb{R}^N)}^p \leq C \|u\|_L^p,
\end{equation}
where $C = \int_0^\infty r^{N-1} w(r) G(r) dr$. In particular, when $v,w \in L^1(0, \infty)$, the embedding $\mathcal{D}_L^{1,p}(\mathbb{R}^N) \hookrightarrow L_K^p(\mathbb{R}^N)$ is compact.
\end{proposition}
\begin{proof}
For any $r \in (0, \infty)$, $\omega \in \mathbb{S}^{N-1}$ (the standard sphere in dimension $N$) and $u \in C_0^\infty (\mathbb{R}^N)$, we write
\begin{equation}
u(r, \omega) = - \int_r^\infty \frac{\partial u}{ \partial t} (t, \omega) dt = - \int_r^\infty t^{\frac{1-N}{p}} v(t)^{\frac{-1}{p}} t^{\frac{N-1}{p}}v(t)^{\frac{1}{p}} \frac{\partial u}{ \partial t} (t, \omega) dt .
\end{equation}
Appling H\"older inequality, we get
\begin{equation*}
|u(r, \omega)| \leq \left(\int_r^\infty t^{\frac{(1-N)}{p-1}} |v(t)|^{\frac{-1}{p-1}} dt \right)^{1/p'} \left(\int_r^\infty t^{N-1} |v(t)| |\frac{\partial u}{ \partial t} (t, \omega)|^p dt \right)^{1/p}    \\
\end{equation*}
Hence
\begin{equation}
|u(r, \omega)|^p \leq G(r) \int_r^\infty t^{N-1} |v(t)| |\frac{\partial u}{ \partial t} (t, \omega)|^p dt .
\end{equation}
Integrating the above over $\mathbb{S}^{N-1}$, we get
\begin{align*}
\int_{S_{N-1}} |u(r, \omega)|^p d \omega &\leq  G(r)  \int_{S_{N-1}} \int_r^\infty t^{N-1} |v(t)| |\frac{\partial u}{ \partial t} (t, \omega)|^p dt d\omega \\
& \leq  G(r) \int_{\mathbb{R}^N} v(|x|) |\nabla u|^p dx .
\end{align*}
Now multiplying by $r^{N-1} w(r)$ and integrating from $0$ to $\infty$, we get
\begin{equation}
\int_0^\infty \int_{S_{N-1}} r^{N-1} w(r) |u(r, \omega)|^p d \omega dr \leq \left( \int_0^\infty r^{N-1} w(r) G(r) dr \right) \left(\int_{\mathbb{R}^N} v(|x|) |\nabla u|^p dx \right) .
\end{equation}
That is,
\begin{equation} \label{embedding}
\int_{\R^N}  w(|x|) |u|^p dx \leq C \int_{\mathbb{R}^N} v(|x|) |\nabla u|^p dx ,
\end{equation}
where $C = \int_0^\infty r^{N-1} w(r) G(r) dr $. By the assumptions on the functions $L$ and $K$, (\ref{Equation 5}) follows. 

Now we prove that the above embedding is compact when $v, w \in L^1(0, \infty)$. Let $u_n$ be a sequence in $\D_L^{1,p}(\mathbb{R}^N)$ such that $u_n \w u$ in  $\D_L^{1,p}(\R^N)$. We show that, up to a subsequence, $u_n \rightarrow u$ in $L_K^p(\R ^N)$. Since $r^{N-1} w(r) G(r) \in L^1(0, \infty)$ and $C_0^\infty (0, \infty)$ is dense in $L^1(0, \infty)$, there is a sequence $G_\epsilon \in C_0^\infty (0, \infty)$ such that $G_\epsilon \rightarrow r^{N-1} w(r) G(r)$ in $L^1(0, \infty)$. That is,
\begin{equation*}
\int_0^\infty |r^{N-1} w(r) G(r) - G_\epsilon (r)| dr < \epsilon.    
\end{equation*}
Let $H_\epsilon(r) = (r^{N-1} G(r))^{-1} G_\epsilon (r)$. Then,
\begin{equation*}
\int_{\R^N}  w(|x|) |u_n- u|^p dx \leq \int_{\R ^N} |(w - H_\epsilon)(|x|)| |u_n- u|^p dx + \int_{\R ^N} |H_\epsilon(|x|)| |u_n -u|^p dx .
\end{equation*}
Note that since the sequence $(u_n)$ is weakly convergent, it is bounded in $\D_L^{1,p}(\R^N)$. Hence, there exists a constant $C>0$ such that $\| u_n - u \|_L \leq C$ for all $n \in \mathbb{N}$. Using this and (\ref{embedding}), we get
\begin{align*}
\int_{\R ^N} |(w - H_\epsilon)(|x|)| |u_n -u|^p dx &\leq  \| u_n - u \|_L^p \int_0^\infty | (r^{N-1} w(r) - r^{N-1}H_\epsilon(r)) G(r) |  dr  \\
&\leq C^p \int_0^\infty |r^{N-1} w(r) G(r) - G_\epsilon (r)| dr \\
&< C^p \epsilon. \numberthis \label{Equation 3}
\end{align*}
Now note that $\D_L^{1,p}(\R^N) \subset W_{loc} ^{1,p} (\R^N)$ and $G_\epsilon \in C_0^\infty(0, \infty)$. Hence, by using the Rellich-Kondrasov compactness theorem, we get a subsequence, which will again be denoted by $u_n$, such that  $\int_{\R ^N} | H_\epsilon(|x|)| |u_n -u|^p dx \rightarrow 0 $ as $n \rightarrow \infty$. Hence
\begin{equation}\label{Equation 4}
\int_{\R ^N} |H_\epsilon(|x|)| |u_n -u|^p dx < \epsilon .
\end{equation}
Combining (\ref{Equation 3}) and (\ref{Equation 4}), we get that, upto a subsequence, $\int_{\R^N}  w(|x|) |u_n- u|^p dx \rightarrow 0$ as $n \rightarrow \infty$.  
\end{proof} 

Next, we recall the following generalization of the C-K-N inequity (see \cite{caldiroli2000variational}).
\begin{lemma} \label{C-K-N 2}
Let the weight function $L$ satisfies (C1)-(C2) . Then there exists a constant $C_\alpha>0$ such that for all $u \in C_0^\infty(\R^N)$,
\begin{equation} 
\left(\int_{\mathbb{R}^N} |u|^{p_\alpha^\ast} dx \right)^{p/p_\alpha ^\ast} \leq C_\alpha \int_{\mathbb{R}^N} L(x) |\nabla u |^p dx ,
\end{equation}
where $p_\alpha^\ast = \frac{np}{n-p- p\alpha}$.
\end{lemma}
We also recall Picone's identity \cite{MR1618334}.
\begin{lemma} \label{Picone's identity}
Let $u,v \in C_0^1(\R^N)$ such that $u\geq 0$, $v>0$. Then 
$$ |\nabla u|^p - |\nabla v|^{p-2} \nabla v \cdot \nabla (\frac{u^p}{v^{p-1}}) \geq 0 \mbox{ in } \R^N.$$
\end{lemma}
\section{Proofs of Theorems \ref{existence}-\ref{asymptotic}}\label{Sec3}
In this section, we provide a proof of the existence, regularity and asymptotic properties of the principal eigenfunctions. 
\subsection{Proof of Theorem \ref{existence}}
We define the functionals $G,I : \D_L^{1,p}(\R^N) \rightarrow \R$ by, 
\begin{equation}\label{defn:G}
G(u) = \int_{\R^N} K(x) |u|^p dx , 
\end{equation}
\begin{equation}\label{defn:I}
I(u) = \int_{\R^N} L(x) |\nabla u|^p dx , 
\end{equation}
and the set $M \subset \D_L^{1,p}(\R^N)$ by,
$$M = \{ u \in \D_L^{1,p}(\R^N) : G(u) =1 \}.$$
Consider the minimisation problem
\begin{equation}\label{Minimisation}
\inf \{ I(u) : u \in M \}.
\end{equation}
Let $\Omega$ be an open set of $\mathbb{R}^N$ such that $K$ is positive on $\Omega$ and let $v \neq 0$ be in $C_0^\infty (\mathbb{R}^N)$ with support in $\Omega$. For such a function $v$, $\int_{\mathbb{R}^N} K(x) |v|^p dx >0$. By replacing $v$ with $\frac{v}{(\int_{\R^N} K(x) |v|^p dx)^{1/p}}$, we get $G(v)=1$. Hence, the set $M$ is nonempty and the infinimum defined in (\ref{Minimisation}) is finite. It is easy to see that the functionals $G,\ I$ are $C^1$ and $I$ are bounded below on the set $M$. Since the embedding $\D_L^{1,p}(\R^N) \subset L_K^p(\R^N)$ is compact, it also follows that the functional $G$ is compact. Hence, by the Lagrange multiplier rule, the infimum defined in (\ref{Minimisation}) is an eigenvalue of the problem (\ref{Pblm 1}). 

We will prove that this infimum is attained on $M$. Let $(u_k)$ be a minimizing sequence in $M$. Since $\D_L^{1,p}(\R^N)$ is reflexive and $(u_k)$ is bounded in $\D_L^{1,p}(\R^N)$, it follows that $(u_k)$ has a subsequence, again denoted by $(u_k)$, which converges weakly to some $u$ in $\D_L^{1,p}(\R^N)$. Since the map $G$ is compact, we get $G(u) = \displaystyle \lim_{k \rightarrow \infty } G(u_k) = 1$. Hence, $u \in M$. Denoting by $\lambda_1$ the infimum in (\ref{Minimisation}) and using the weak lower semicontinuity of norm, we get 
\begin{equation*}
\lambda_1 \leq I(u) \leq \liminf_{k \rightarrow \infty} I(u_k) = \lambda_1,
\end{equation*}
which completes our proof. If $u$ is a critical point of $I$ on $M$, then $|u|$ is also a critical point. That is, there exists a nonnegative eigenfunction corresponding to $\lambda_1$. 

We now prove the positivity of the nonnegative eigenfunctions of (\ref{Pblm 1}). We use the following version of the strong maximum principle proved for the weak solutions of (\ref{Pblm 1}) by Kawohl et al. in \cite{MR2305874}. Here, the authors have considered only the bounded domains in $\mathbb{R}^N$. However, their results hold in our case by restricting the solutions of (\ref{Pblm 1}) to bounded subdomains of $\mathbb{R}^N$. 
\begin{lemma}\label{lemma 1}
Let $0 \leq V \in L_{loc}^1(\mathbb{R}^N)$. Let $u \in \mathcal{D}_L^{1,p}(\mathbb{R}^N)$  be a non negative function such that $V |u|^p \in L_{loc}^1(\mathbb{R}^N)$ and satisfies:
\begin{equation}
\int_{\mathbb{R}^N} L(x) |\nabla u|^{p-2} \nabla u \nabla v dx + \lambda \int_{\mathbb{R}^N} V(x) |u|^{p-2} u v dx \geq 0 , \mbox{ for all }  v \in C_0^\infty(\mathbb{R}^N), v \geq 0 .
\end{equation}
Then either $u \equiv 0$ or else the set $\{ x \in \mathbb{R}^N : u(x) = 0 \}$ is of measure zero.
\end{lemma}
Now, let $u \in \D_L^{1,p}(\R^N)$ be a nonnegative eigenfunction corresponding to $\lambda_1$. Then, for all  $v \in C_0^\infty(\mathbb{R}^N)$, $v \geq 0$,
\begin{align*}
\int_{\mathbb{R}^N} L(x) |\nabla u|^{p-2} \nabla u \nabla v dx + \lambda_1 \int_{\mathbb{R}^N} K^-(x) |u|^{p-2} u v dx = \lambda_1 \int_{\mathbb{R}^N} K^+(x) |u|^{p-2} u v dx \geq 0.
\end{align*}
Hence, by Lemma \ref{lemma 1}, either $u \equiv 0$ or else the set $\{ x \in \mathbb{R}^N : u(x) = 0 \}$ is of measure zero. \qed 

This completes the proof of Theorem \ref{existence}. Next, we show the regularity of the first eigenfunctions. 
\subsection{Proof of Theorem \ref{regularity}}
We prove that the first eigenfunction is in $L^\infty(\mathbb{R}^N)$. Our proof uses Moser-type iterations as given in \cite{Drabek-95}. For each $k \in \mathbb{N}$, set $u_k(x) = \min \{ u(x), k \}$. Then $u_k \in \D_L^{1,p}(\R^N) \cap L^\infty (\R^N)$, for each $\beta>0$, $u_k^{p \beta+1} \in \D_L^{1,p}(\R^N) \cap L^\infty (\R^N)$ and we have
\begin{equation} \label{eqn.regularity.1}
\int_{\mathbb{R}^N} L(x) |\nabla u|^{p-2} \nabla u \nabla u_k^{p \beta+1} dx = \lambda_1 \int_{\mathbb{R}^N} K(x) |u|^{p-2} u u_k^{p \beta+1} dx.
\end{equation}
\begin{align*}
\int_{\mathbb{R}^N} L(x) |\nabla u|^{p-2} \nabla u \nabla u_k^{p \beta+1} dx &=  ( p \beta +1) \int_{\mathbb{R}^N} L(x)  |\nabla u_k |^p |u_k|^{p \beta} dx \\
&= \frac{(p \beta+1)}{(\beta+1)^p} \int_{\mathbb{R}^N} L(x) | \nabla u_k ^{\beta+1}|^p dx \\
& \geq \frac{1}{C_\alpha} \frac{(p \beta+1)}{(\beta+1)^p} \left(\int_{\mathbb{R}^N} |u^{\beta +1 }|^{p_\alpha^\ast } dx \right)^{p/p_\alpha^\ast},
\end{align*}
where the last inequality follows from Lemma \ref{C-K-N 2}. We also have the following estimate for the right-hand side of (\ref{eqn.regularity.1}).
\begin{align*}
\lambda_1 \int_{\mathbb{R}^N} K(x) |u|^{p-2} u u_k^{p \beta+1} dx & \leq \lambda_1 \int_{\mathbb{R}^N} |K(x)| |u|^{p-2} u u^{p \beta+1} dx \\
& \leq \lambda_1 \| K \|_\infty  \| u \|_{p(\beta+1)}^{p(\beta+1)}.
\end{align*}
Hence, it follows that 
\begin{equation}
\| u \|_{(\beta +1) p_\alpha^\ast} ^{p(\beta+1)} \leq \lambda_1 C_\alpha \frac{(\beta+1)^p}{(p \beta +1)} \| K \|_{\infty} \| u \|_{p(\beta+1)}^{p(\beta+1)}.
\end{equation}
Let $C_\beta = [\lambda_1C_\alpha \frac{(\beta+1)^p}{(p \beta +1)} \| K \|_{\infty}] ^{1/p(\beta+1)} $, then
\begin{equation} \label{eqn.regularity.2}
\| u \|_{(\beta +1) p_\alpha^\ast} \leq C_\beta \| u \|_{p(\beta+1)}.
\end{equation}
Note that (\ref{eqn.regularity.2}) is true for any $\beta >0$. Now choose $\beta_1>0$ so that $p(\beta_1+1) = p_\alpha^\ast$, or $\beta_1+1 = \frac{p_\alpha^\ast}{p} =a$. Then
\begin{equation}
\| u \|_{p_\alpha^\ast a} \leq C_{\beta_1} \| u \|_{p_\alpha^\ast}.
\end{equation}
For any $n \in \mathbb{N}$, we choose $\beta_n>0$ so that $\beta_n +1 = a^n$. Therefore,
\begin{equation} \label{eqn.regularity.3}
\| u \|_{p_\alpha ^\ast a^n} \leq C_{\beta_n} \| u \|_{p_\alpha ^\ast a^{n-1}} \leq (\Pi_{i=1}^n C_{\beta_i}) \| u \|_{p_\alpha^\ast}.
\end{equation}
Which implies that $u \in L^s(\mathbb{R}^N)$ with $s = p_\alpha^\ast a^n$, $n \in \mathbb{N}$. Since $a = \frac{p_\alpha ^\ast}{p}>1$, we have that $a^n \rightarrow \infty$ and hence by interpolation we get $u \in L^s(\mathbb{R}^N)$, for all $p_\alpha^\ast \leq s < \infty$. Also, 
\begin{align*}
C_{\beta_n} &= [\lambda_1 C_\alpha \frac{(\beta_n +1)^p}{(p \beta_n +1)} \| K \|_{\infty}] ^{1/p(\beta_n+1)} \\
&< (\lambda_1 C_\alpha)^{1/p(\beta_n +1)} (a^n)^{1/{a^n}} \| K \|_\infty ^{1/p(\beta_n +1)} \\
&< C^{1/a^n} (a^n)^{1/{a^n}},
\end{align*}
where $C$ is a constant independent of $a$ and $n$. Therefore we get
\begin{equation}
\Pi_{i=1} ^\infty C_{\beta_i} \leq C^{\sum_{i=1}^\infty 1/a^i} e ^ {\sum_{i=1} ^ \infty 1/a^i   \log a^i} < \infty.
\end{equation}
Taking limit in (\ref{eqn.regularity.3}), we get $u \in L^\infty(\mathbb{R}^N)$. By the regularity results in \cite{MR0709038}, it also follows that $u \in C_{loc}^{1, \alpha}(\mathbb{R}^N)$, for some $\alpha \in (0,1)$. This completes the proof.
\hfill $\qed$ 

Next, we proceed with proving the boundedness of the first (radial) eigenfunctions.\vspace{-1em}
\subsection{Proof of Theorem \ref{asymptotic}}
Let $u$ be the principal eigenfunction of (\ref{Pblm 1}) such that $u>0$ and $\int_{\R^2} K(x) |u|^2 dx =1$.
Assume that the functions $K$ and $L$ are radial and positive. The simplicity of the first eigenvalue shows that  $u$  is also radial. Consequently, for $r>0$, (\ref{Pblm 1}) can be written as,
\begin{equation} \label{Pblm Radial}
[r L(r) u'(r)]' = - \lambda_1 r K(r) u(r).
\end{equation}
Let $r_0>0$ be fixed. Integrating (\ref{Pblm Radial}) from $r_0$ to $r$, we get
\begin{equation} \label{Equation 2}
rL(r)u'(r) = r_0 L(r_0) u'(r_0) - \lambda_1 \int_{r_0} ^r s K(s) u(s) ds.
\end{equation}
Since $K$ and $u$ are positive, from (\ref{Pblm Radial}) it follows that $r L(r)u'(r)$ is decreasing. Since $u\in \mathcal{D}_L^{1,2}(\mathbb{R}^2)$, $\int_0^\infty r L(r) |u'(r)|^2 dr < \infty$. Hence, $r L(r)u'(r)$ must decrease to zero (if not, there exist $c>0$ and $r_1>0$ such that $rL(r)u'(r) \geq c$ for all $r>r_1$, which contradicts the integrability of $rL(r)|u'(r)|^2$). Letting $r \rightarrow \infty$ in (\ref{Equation 2}), we get
\begin{equation}
r_0 L(r_0) u'(r_0) = \lambda_1 \int_{r_0} ^\infty s K(s) u(s) ds.
\end{equation}
Therefore from (\ref{Equation 2}) we have
\begin{equation} \label{Equation 6}
r L(r) u'(r) = \lambda_1 \int_r^\infty s K(s) u(s) ds,
\end{equation}
which implies that $u'(r)>0$ and $u$ is a radially increasing function (Hence, $u$ cannot decay to zero). 

Now we prove that with the additional assumptions in Theorem \ref{asymptotic}, $u$ is bounded from above. We have
\begin{equation}
[r L(r) u'(r) F(r)]' = [r L(r) u'(r)]' F(r) + u'(r).
\end{equation}
Using (\ref{Pblm Radial}), the above equation can be rewritten as
\begin{equation}
[u(r) - r L(r) u'(r) F(r)]' = \lambda_1 r K(r) u(r) F(r).
\end{equation}
Integrating this from 0 to $r$ and using (\ref{Equation 6}), we get
\begin{equation}
u(r) = u(0) + \lambda_1 F(r) \int_r ^ \infty s K(s) u(s) ds + \lambda_1 \int_0^r s K(s) u(s)F(s) ds .
\end{equation}
Note that $F$ is an increasing function, and hence
\begin{align*}
u(r) &< u(0) + \lambda_1 \int_0^\infty s F(s) K(s) u(s) ds \\
& \leq u(0) + \lambda_1 (\int_0^\infty s [F(s)]^2 K(s) ds)^{1/2} (\int_0^\infty s K(s) |u(s)|^2 ds )^{1/2} \\
&= u(0) + \lambda_1 (\int_0^\infty s [F(s)]^2 K(s) ds)^{1/2} \int_{\R^2} K(x) |u|^2 dx.
\end{align*}
Hence, from our assumption on $K$, it follows that $u$ is bounded.
\section{Antimaximum Principle}\label{Sec4}
 In this section, we prove a local and global antimaximum principle for the perturbed problem (\ref{Pblm AMP}). Note that by Theorem \ref{existence}, the first eigenvalue $\lambda_1$ of (\ref{Pblm 1}) is given by,
\begin{equation}\label{eigenvalue}
\lambda_1 = \inf \left\{ I(u) : u \in \D_L^{1,p}(\R^N), G(u) = 1 \right\},
\end{equation}
where $G$, $I$ are as defined in (\ref{defn:G}) and (\ref{defn:I}) respectively. Let $\Phi_1$ be the principal eigenfunction corresponding to $\lambda_1$ such that $\lambda_1 = \int_{\R^N} L(x) |\nabla \Phi_1|^p dx $. 
It can be shown by using Picone's identity (see Lemma \ref{Picone's identity}) that $\lambda_1$ is unique. We now establish some results that will help us to prove the antimaximum principle.   
\begin{proposition} \label{Proposition AMP 1}
Let $h$ satisfy $h \geq 0$, $h \not \equiv 0$. If $\lambda \in (0, \lambda_1)$ then any solution $u$ of (\ref{Pblm AMP}) is positive on $\R^N$. 
\end{proposition}
\begin{proof}
We have for all $v \in \D_L^{1,p}(\R^N)$, 
\begin{equation} \label{Equation 7}
\int_{\R^N} L(x) |\nabla u|^{p-2} \nabla u \nabla v dx = \lambda \int_{\R^N} K(x) |u|^{p-2}  uv dx + \int_{\R^N} h(x) v dx .
\end{equation}
We will prove that $u^- = 0$ a.e. in $\R^N$. If not, take the test function in (\ref{Equation 7}) to be $v = -u^-$, then 
\begin{align*}
\int_{\R^N} L(x) |\nabla u^-|^p dx & = \lambda \int_{\R^N} K(x) |u^-|^p - \int_{\R^N} h(x) u^- \\
& \leq \lambda \int_{\R^N} K(x) |u^-|^p \\
& = \lambda G(u^-) .
\end{align*}
We have $\lambda_1 = \inf \{ \int_{\R^N} L(x) |\nabla u|^P ; u \in \D_L^{1,p}(\R^N), G(u) =1 \} $. Let $w = \frac{u^-}{G(u^-)^{1/p}}$. Then, $G(w) =1$ and 
\begin{align*}
\lambda_1 & < \int_{\R^N} L(x) |\nabla w|^p dx \\
&= \frac{1}{G(u^-)} \int_{\R^N} L(x) |\nabla u^-|^p dx \\
& \leq \frac{1}{G(u^-)} \lambda G(u^-) = \lambda < \lambda_1,
\end{align*}
which is a contradiction. Hence, $u^-=0$ a.e. Now, $u>0$ will follow from the maximum principle. 
\end{proof}

Using Picone's identity, we can prove the following non-existence result when $\lambda \geq \lambda_1$.
\begin{proposition} \label{Proposition AMP 2}
Let $h$ satisfy $h \geq 0$, $h \not \equiv 0$. Then (\ref{Pblm AMP}) has no solution if $\lambda= \lambda_1$, and no nonnegative solution if $\lambda>\lambda_1$. 
\end{proposition}
\begin{proof}
Assume that there exists a solution $u$ when $\lambda = \lambda_1$. Following the same approach as in the proof of Proposition \ref{Proposition AMP 1}, we obtain $u^- =0$ almost everywhere. Now, suppose that (\ref{Pblm AMP}) admits a positive solution $u$ when $\lambda \geq \lambda_1$. Let $u_k \in C_0^1(\R^N)$ be such that $u_k \geq 0$ and $u_k \rightarrow \Phi_1$ in $D_L^{1,p}(\R^N)$. Applying Picone's identity to the functions $u$ and $u_k$ we get 
\begin{align*}
0 & \leq \int_{\R^N} L(x) |\nabla u_k|^p dx - \int_{\R^N} L(x) |\nabla u|^{p-2} \nabla u \cdot \nabla\left( \frac{u_k^p}{u^{p-1}} \right) dx \\
&= \int_{\R^N} L(x) |\nabla u_k|^p dx -\lambda \int_{\R^N} K(x) |u_k|^p dx - \int_{\R^N} h(x) \frac{u_k^p}{u^{p-1}} dx .
\end{align*}
Now applying Fatou's lemma, we get
$$\int_{\R^N} h(x) \frac{\Phi_1^p}{u^{p-1}} \leq (\lambda_1 - \lambda) \int_{\R^N} K(x) \Phi_1^p dx = \lambda_1 - \lambda,$$ which gives a contradiction since the left hand side is positive and the right hand side is nonpositive. 
\end{proof} 
The following proposition establishes that the first eigenvalue $\lambda_1$ is isolated. The proof follows with minor modifications from Lemma 2.3 in \cite{Dr-Hu}. Therefore, we omit the details. 
\begin{proposition}\label{Proposition Isolated}
Let $K$ and $L$ be as given in Theorem \ref{existence} and let $\lambda_1$ be the first eigenvalue of the problem (\ref{Pblm 1}). Then there exists $\epsilon>0$ such that (\ref{Pblm 1}) has no eigenvalue in $(\lambda_1, \lambda_1 + \epsilon)$.   
\end{proposition}
Using Proposition \ref{Proposition Isolated} and the Fredholm alternative for nonlinear operators established by Fu\v{c}\'ik et al. (See \cite[Theorem 3.2]{MR467421}), we can prove the following existence result.
\begin{proposition}
Let $K$ be as stated in Theorem \ref{existence} and $\lambda_1$ be the first eigenvalue of the problem (\ref{Pblm 1}). Then there exists $\epsilon>0$ such that (\ref{Pblm AMP}) has a solution in $\D_L^{1,p}(\R^N)$ for every $\lambda \in (0, \lambda_1+\epsilon) \setminus \{ \lambda_1 \} $ and $h$ in the dual of $\D_L^{1,p}(\R^N)$.
\end{proposition}

Now we will prove the local and global antimaximum principles. \\
\textbf{Proof of Theorem \ref{Theorem: local AMP}:} Assume, for contradiction, that the conclusion in Theorem \ref{Theorem: local AMP} is not true. Then for each $k \in \mathbb{N}$, there exist $\mu_k, \ u_k$ such that $\mu_k > \lambda_1$, $\mu_k \downarrow \lambda_1$, $u_k$ is a solution to
$$ - \mbox{div} (L(x) |\nabla u_k|^{p-2} \nabla u_k) = \mu_k K(x) |u_k|^{p-2} u_k + h(x) \mbox{ in } \R^N,$$ and $|E \cap  \{ u_k \geq 0 \}| >0$ for all $k$. We will first prove that $\lim_{k \rightarrow \infty} \| u_k\|_L = \infty$. If this is not true, then $(u_k)$ has a subsequence, again denoted by $(u_k)$, such that $u_k \rightharpoonup u $ in $\D_L^{1,p}(\R^N)$. By Serrin $L^\infty$ estimate and Tolksdorff regularity result (see \cite{Serrin}, \cite{Tolksdorf}), $u_k$ is bounded in $C^{1,\alpha}(E)$. Hence, by Arzela Ascoli, $u_k \rightarrow u$ in $C^1_{loc}(\R^N)$. Letting $k \rightarrow \infty$ and using the density of $C_0^\infty(\R^N)$ in $D_L^{1,p}(\R^N)$, we get 
$$- \mbox{div } (L(x) |\nabla u|^{p-2} \nabla u) = \lambda_1 K(x) |u|^{p-2} u + h(x) \mbox{ in } \R^N, $$ which is a contradiction to Proposition \ref{Proposition AMP 2}. Now, set $v_k = \frac{u_k}{\| u_k \|_L}$. Then for all $\phi \in \D_L^{1,p}(\R^N)$,
\begin{align*}
\int_{\R^N} L(x) |\nabla v_k|^{p-2} \nabla v_k \nabla \phi &= \frac{1}{\| u_k \|_L^{p-1}} \int_{\R^N} L(x) |\nabla u_k|^{p-2} \nabla u_k \nabla \phi \\
&= \frac{1}{\|u_k\|_L^{p-1}} \left[ \mu_k \int_{\R^N} K(x) |u_k|^{p-2} u_k \phi + \int_{\R^N} h(x) \phi \right] \\
& = \mu_k \int_{\R^N} K(x) |v_k|^{p-2} v_k \phi + \frac{1}{ \| u_k \|_L^{p-1}} \int_{\R^N} h(x) \phi \numberthis \label{equation: local AMP}.
\end{align*}
Since $(v_k)$ is bounded there exists a subsequence, again denoted by $v_k$, such that $v_k \rightharpoonup v$, $v_k \rightarrow v$ in $C_{loc}^1(\R^N)$. Letting $k \rightarrow \infty$ in (\ref{equation: local AMP}), from the compactness of $G$, we get 
\begin{equation*}
\int_{\R^N}  L(x) |\nabla v|^{p-2} \nabla v \nabla \phi = \lambda_1 \int_{\R^N}  K(x) |v|^{p-2} v \phi, \mbox{ for all } \phi \in \D_L^{1,p}(\R^N) .  
\end{equation*}
Since the first eigenvalue $\lambda_1$ is simple, it follows that $v = c \Phi_1$ for some constant $c$. We also have 
\begin{equation*}
1= \int_{\R^N}  L(x) |\nabla v|^{p} dx = \lambda_1 \int_{\R^N}  K(x) |v|^{p} dx.
\end{equation*}
Hence, $c = \pm (\frac{1}{\lambda_1})^{1/p}$. Now suppose $c = -\lambda_1^{1/p}$. Then $v_k \rightarrow c \Phi_1$ uniformly on $\overbar{E}$ will imply that $u_k = v_k \| u_k \|_L$ is negative on $\overbar{E}$ for sufficiently large values of $k$. This contradicts the fact that $|E \cap  \{ u_k \geq 0 \}| >0$, for all $k$. Therefore $c=\lambda_1^{1/p}$. We claim that $v_k \geq 0$ on $\R^N$ for sufficiently large values of $k$. Suppose, for contradiction, that $v_k^- \not \equiv 0$. Using $- v_k^-$ as a test function, we obtain
\begin{align*}
0 < \int_{\R^N} L(x) |\nabla v_k^-|^p dx & = \mu_k \int_{\R^N} K(x) (v_k^-)^p dx - \frac{1}{\| u_k \|_L^{p-1}} \int_{\R^N} h(x) v_k^- dx \\
& \leq \mu_k \int_{\R^N} K(x) (v_k^-)^p dx .
\end{align*}
Let $w_k = \frac{v_k^-}{G(v_k^-)^{1/p}}$. Then $\int_{\R^N} K(x) w_k^p dx =1$ and $$\int_{\R^N} L(x) |\nabla w_k|^p dx = \frac{1}{G(v_k^-)} \int_{|R^N} L(x) |\nabla v_k^-|^p dx \leq \mu_k .$$ 
This implies $\lambda_1 \leq \int_{\R^N} L(x) | \nabla w_k|^p dx \leq \mu_k$. Using the reflexivity of $\D_L^{1,p}(\R^N)$ and the weak lower semicontinuity of the norm we conclude that, up to a subsequence, $w_k \rightharpoonup w \in \D_L^{1,p}(\R^N)$ and 
\begin{equation*}
\int_{\R^N} L(x) |\nabla w|^p dx \leq \liminf_{k \rightarrow \infty} \int_{\R^N} L(x) |\nabla w_k|^p dx = \lambda_1.
\end{equation*}
Since $\lambda_1$ is simple, we must have $(w_k) \rightarrow \Phi_1$ in $\D_L^{1,p}(\R^N)$. Since $v_k \rightarrow \lambda_1^{1/p} \Phi_1$ in $C_{loc}^1(\R^N)$, we must have $v_k>0$ for large $k$ on $\overbar{E}$. Hence, $w_k=0$ for large $k$ on $\overbar{E}$, leading to a contradiction. 
\hfill $\qed$ 

In \cite{Fleckinger}, the authors established the global antimaximum principle for (\ref{Pblm AMP}) when $1<p<N$, when $K$ has a dominant negative part and $h$ has compact support. Here, we extend the result to include the case $p=N$ without imposing any sign restriction on the weight function $K$. \\
\textbf{Proof of Theorem \ref{Theorem: Global AMP}:} Assume that the global antimaximum principle does not hold. Then for each $k \in \mathbb{N}$ there exist$\mu_k, \ u_k$, as described in the proof of Theorem \ref{Theorem: local AMP}, such that $|\{ u_k \geq 0 \}|>0$ for all $k$. We claim that for large values of $k$, $|\mbox{supp}(u_k^+)|>0$. 

Since $h$ has compact support, there exists $r>0$ such that $\mbox{supp}(h) \subset B(0,r)$. Now, applying the local antimaximum principle, there exists $k_0 \in \mathbb{N}$ sufficiently large so that for all $k \geq k_0$, $u_k<0$ in $\overbar{B(0,r)}$. Suppose for some $\overbar{k} \geq k_0$, $|\mbox{supp}(u_{\overbar{k}}^+)| =0$. Then $v= - u_{\overbar{k}}$  satisfies 
\begin{equation}\label{Equation 8}
- \mbox{div} (L(x) |\nabla v|^{p-2} \nabla v) = \lambda K(x) |v|^{p-2} v - h(x) \mbox{ in } \R^N .    
\end{equation} 
Applying Lemma \ref{lemma 1} and using the fact that $h \equiv 0$ in $B(0,r)^c$, we conclude that either $v \equiv 0$ or $v>0$ in $B(0,r)^c$. Since $v$ is continuous and $v>0$ in $\partial B(0,r)$, it follows that $v=-u_{\overbar{k}}>0$ in $B(0,r)^c$. This, together with the local antimaximum principle, implies that $u_{\overbar{k}}<0$ in $\R^N$, leading to a contradiction. Therefore, $|\mbox{supp}(u_k^+)| >0$ for all $k \geq k_0$. Now, let $w_k = \frac{u_k^+}{G(u_k^+)^{1/p}}$. Then $G(w_k) =1$ and $\int_{\R^N} L(x) |\nabla w_k|^p dx = \mu_k $. Following the same approach as in Theorem \ref{Theorem: local AMP}, we obtain $w_k \rightarrow \Phi_1$ in $\D_L^{1,p}(\R^N)$. But for all $k \geq k_0$, $u_k<0$ in $\overbar{B(0,r)}$ and hence $w_k \equiv 0$, which gives a contradiction.
\hfill $\qed$

\section*{Acknowledgements}
AS expresses his sincere gratitude to DST-INSPIRE with Grant DST/INSPIRE/04/  2018/002208, funded by the Government of India. The research for AJ was partially conducted during the author's visit to IIT Jodhpur, funded by DST/INSPIRE/04/2018/  002208.

\bibliography{References}
\bibliographystyle{amsplain}
	
\end{document}